\documentclass[11pt]{amsart}
\usepackage{amsmath, amssymb, verbatim}
\usepackage{pinlabel}
\usepackage{graphs}
\usepackage{thmtools, thm-restate}

\usepackage{graphicx}
\usepackage{color}
\usepackage[margin=2.5cm]{geometry}                
\newtheorem{thm}{Theorem}
\newtheorem{cor}[thm]{Corollary}

\newtheorem{prop}[thm]{Proposition}

\theoremstyle{definition}
\newtheorem{defn}{Definition}

\newcommand{\xra}{\xrightarrow}

\newcommand{\si}{\sigma}

\newcommand{\la}{\lambda}
\newcommand{\tla}{\tilde\lambda}
\newcommand{\Si}{\Sigma}

\newcommand{\tB}{\widetilde{B}}

\renewcommand{\t}{\mathbf t}

\newcommand{\s}{\mathbf s}

\newcommand{\Z}{\mathbb Z}

\newcommand{\wHF}{\widehat{HF}}

\renewcommand{\bar}{\overline}

\title{Signatures, Heegaard Floer correction terms and quasi--alternating links}
\author{Paolo Lisca}
\author{Brendan Owens}
\thanks{The present work is part of the first author's activities within CAST, 
a Research Network Program of the European Science Foundation, and the PRIN--MIUR research project 
2010--2011 ``Variet\`a reali e complesse: geometria, topologia e analisi armonica''.
The second author was supported in part by  EPSRC grant EP/I033754/1.}

\begin{document}
\begin{abstract}
Turaev showed that there is a well--defined map assigning to an oriented link $L$ in the three--sphere 
a Spin structure $\t_0$ on $\Si(L)$, the two--fold cover of $S^3$ branched along $L$.  We prove, generalizing results 
of Manolescu--Owens and Donald--Owens, that for an oriented quasi--alternating link $L$ the signature of $L$ equals minus four 
times the Heegaard Floer correction term of $(\Si(L), \t_0)$.
\end{abstract}

\maketitle

\section{Introduction}\label{s:intro}
Vladimir Turaev~\cite[\S~2.2]{Tu88} proved that there is a surjective map which associates to a link $L\subset S^3$ decorated with an 
orientation $o$ a Spin structure $\t_{(L,o)}$ on $\Si(L)$, the double cover of $S^3$ branched along $L$. Moreover, he showed 
that the only other orientation on $L$ which maps to $\t_{(L,o)}$ is $-o$, the overall reversed orientation. In other words, 
Turaev described a bijection 
between the set of quasi--orientations on $L$ (i.e.~orientations up 
to overall reversal) and the set Spin$(\Si(L))$ of Spin structures on $\Si(L)$. Each element $\t\in{\rm Spin}(\Si(L))$ 
can be viewed as a Spin$^c$ structure on $\Si(L)$, so if $\Si(L)$ is a rational homology sphere it 
makes sense to consider the rational number $d(\Si(L),\t)$, where $d$ is the correction term invariant defined by Ozsv\'ath and Szab\'o~\cite{OSz03}. 
Under the assumption that $L$ is nonsplit alternating it was proved --- in~\cite{MOw07} when $L$ is a knot 
and in~\cite{DO12} for any number of components of $L$ --- that 
\begin{equation}\label{e:main}\tag{$*$}
\si(L,o) = -4 d(\Si(L), \t_{(L,o)})\quad\text{for every orientation $o$ on $L$},
\end{equation}
where $\si(L,o)$ is the link signature.  
For an alternating link associated to a plumbing graph with no bad vertices,
this follows from a combination of earlier results of Saveliev \cite{Sa00} and Stipsicz \cite{St08}, each of whom 
showed that one of the quantities in \eqref{e:main} is equal to  the Neumann-Siebenmann $\overline\mu$-invariant of the plumbing tree.
The main purpose of this paper is to prove Property~\eqref{e:main} for the family of 
{\em quasi--alternating links} introduced in~\cite{OSz05}: 

\begin{defn}\label{d:qa}
The {\em quasi--alternating} links are the links in $S^3$ with nonzero determinant defined recursively as follows: 
\begin{enumerate}
\item
the unknot is quasi--alternating; 
\item
if $L_0, L_1$ are quasi--alternating, $L\subset S^3$ is a link such that $\det L = \det L_0 + \det L_1$ and 
$L$, $L_0$, $L_1$ differ only inside a 3--ball as illustrated in Figure~\ref{f:skein}, then $L$ is quasi--alternating. 
\end{enumerate}
\begin{figure}[ht]
\labellist\hair 2pt
\pinlabel $L_1$ at 40 -10
\pinlabel $L$ at 193 -10
\pinlabel $L_0$ at 354 -10
\endlabellist
\centering
\includegraphics[scale=0.6]{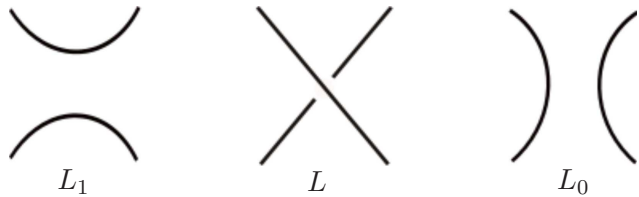}
\vspace{0.2cm}
\caption{$L$ and its resolutions $L_0$ and $L_1$.}
\label{f:skein}
\end{figure}
\end{defn}
Quasi--alternating links have recently been the object of considerable 
attention~\cite{CK09, CO12, Gr10, GW11, JS08, MO07, QCQ12, QQJ12, Wa09, Wi09}. Alternating links are 
quasi--alternating~\cite[Lemma~3.2]{OSz05}, but (as shown in e.g.~\cite{CK09}) there exist infinitely 
many quasi--alternating, non--alternating links. Our main result is the following:
\begin{restatable}{theorem}{main}\label{t:main2} 
Let $(L,o)$ be an oriented link. If $L$ is quasi--alternating then
\begin{equation}\label{e:main2}
\si(L,o) = -4 d(\Si(L), \t_{(L,o)}).
\end{equation}
\end{restatable}
The contents of the paper are as follows. In Section~\ref{s:prelim} we first recall some basic facts on Spin structures and 
the existence of two natural 4--dimensional cobordisms, one from $\Si(L_1)$ to $\Si(L)$, the other from 
$\Si(L)$ to $\Si(L_0)$. Then, in Proposition~\ref{p:spincob} we show that for an orientation $o$ on $L$  for which
the crossing in Figure~\ref{f:skein} is positive, the Spin structure $\t_{(L,o)}$ extends to the first cobordism 
but not to the second one. In Section~\ref{s:correction} we use this information together with the Heegaard Floer surgery exact triangle 
to prove Proposition~\ref{p:correction}, which relates the value of the correction term $d(\Si(L),\t_{(L,o)})$ 
with the value of an analogous correction term for $\Si(L_1)$. In Section~\ref{s:final}  we restate and prove our main result,  
Theorem~\ref{t:main2}. The proof consists of an inductive argument based on Proposition~\ref{p:correction} 
and the known relationship between the signatures of $L$ and $L_1$. The use of Proposition~\ref{p:correction} 
is made possible by the fact that up to mirroring $L$ one may always assume the crossing of Figure~\ref{f:skein} to 
be positive. We close Section~\ref{s:final} with Corollary~\ref{c:jones-turaev}, which uses results of Rustamov and Mullins to relate 
Turaev's torsion function for the two--fold branched cover of a quasi--alternating link $L$ with the Jones polynomial of $L$. 

\vskip2mm\noindent{\bf Acknowledgements.} The authors would like to thank the anonymous referee 
for suggestions which helped to improve the exposition.


\section{Triads and Spin structures}\label{s:prelim} 

A Spin structure on an $n$--manifold $M^n$ is a double cover of the oriented frame bundle of $M$ with the added condition that if $n>1$, it 
restricts to the nontrivial double cover on fibres.  A Spin structure on a manifold restricts to give a Spin structure on a codimension--one 
submanifold, or on a framed submanifold of codimension higher than one.
As mentioned in the introduction, an orientation $o$ on a link $L$ in $S^3$ induces a Spin structure $\t_{(L,o)}$ on the double--branched 
cover $\Si(L)$, as in \cite{Tu88}.  Recall also that there are two Spin structures on $S^1=\partial D^2$: the nontrivial or \emph{bounding} 
Spin structure, which is the restriction of the unique Spin structure on $D^2$, and the trivial or \emph{Lie} Spin structure, which does not 
extend over the disk. The restriction map from Spin structures on a solid torus to Spin structures on its boundary is injective; thus if two Spin structures 
on a closed 3--manifold agree outside a solid torus then they are the same. 
For more details on Spin structures see for example~\cite{Kirby89}.

If $Y$ is a 3--manifold with a Spin structure $\t$ and $K$ is a knot in $Y$ with framing $\lambda$, we may attach a 2--handle to $K$ giving 
a surgery cobordism $W$ from $Y$ to $Y_\lambda(K)$.  
There is a  unique Spin structure on $D^2\times D^2$, which restricts to the bounding Spin structure on each framed circle 
$\partial D^2\times\{\mathrm{point}\}$ in $\partial D^2\times D^2$. Thus the Spin structure on $Y$ extends over $W$ if and only if its restriction to $K$, 
viewed as a framed submanifold via the framing $\lambda$, is the bounding Spin structure.  Note that this is equivalent, symmetrically, 
to the restriction of $\t$ to the submanifold $\lambda$ framed by $K$ being the bounding Spin structure.  
Moreover, the extension over $W$ is unique if it exists.

Let $L$, $L_0$, $L_1$ be three links in $S^3$ differing only in a 3--ball $B$ as in Figure~\ref{f:skein}.  The double cover of $B$ branched 
along the pair of arcs $B\cap L$ is a solid torus $\tB$ with core $C$.  The boundary of a properly embedded disk in $B$ which separates the two branching 
arcs lifts to a disjoint pair of meridians of $\tB$.
The preimage in $\Si(L)$ of the curve $\lambda_0$ shown in Figure \ref{f:lambdas} is a pair of parallel framings for $C$; denote one of 
these by $\tla_0$.  Similarly, let $\tla_1$ denote one of the components of the preimage in $\Si(L)$ of $\lambda_1$.
Since $\lambda_0$ is homotopic in $B-L$ to the boundary of a disk separating the two components of $L_0\cap B$, 
we see that $\Si(L_0)$ is obtained from $\Si(L)$ by $\tla_0$--framed surgery on $C$. Similarly, $\lambda_1$  is homotopic in $B-L$ to the boundary of a disk 
separating the two components of $L_1\cap B$, and $\Si(L_1)$ is obtained from $\Si(L)$ by $\tla_1$--framed surgery on $C$.

The two framings $\tla_0$ and $\tla_1$ differ by a meridian of $C$.  
In the terminology from \cite{OSz05}, the manifolds $\Si(L)$, $\Si(L_0)$ and $\Si(L_1)$ form a \emph{triad} and there are 
surgery cobordisms
\begin{equation}\label{e:surgcobs}
V:\Si(L_1)\to\Si(L),\quad\text{and}\quad W:\Si(L)\to\Si(L_0).
\end{equation}
The surgery cobordism $W$ is built by attaching a 2--handle to $\Si(L)$ along the knot $C$ with framing $\tla_0$.  The cobordism $V$ is built by attaching a 2--handle to $\Si(L_1)$.  Dualising this handle structure, $V$ is obtained by attaching a 2--handle to $\Si(L)$ along the knot $C$ with framing $\tla_1$ (and reversing orientation).

\begin{figure}[ht]
\labellist
\small\hair 2pt
\pinlabel $\lambda_1$ at 17 45
\pinlabel $\lambda_0$ at 80 45
\pinlabel $L$ at 40 0
\endlabellist
\centering
\includegraphics[scale=1]{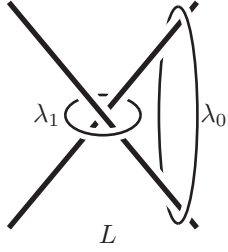}
\caption{The loops $\lambda_0$ and $\lambda_1$.}
\label{f:lambdas}
\end{figure}

\begin{prop}\label{p:spincob}
For any orientation $o$ on $L$ such that the crossing shown in Figure \ref{f:skein} is positive, the Spin structure $\t_{(L,o)}$ extends to a unique Spin structure $\s_o$ on the cobordism $V$ and does not admit an extension over $W$.  The restriction of $\s_o$ to $\Si(L_1)$ is the Spin structure $\t_{(L_1,o_1)}$, where $o_1$ is the orientation on $L_1$ induced by $o$.
\end{prop}
\begin{proof}
Let $\pi:\Si(L)\to S^3$ be the branched covering map.  The Spin structure $\t_{(L,o)}$ is the lift $\tilde\s$ of the Spin structure restricted from $S^3$ to $S^3-L$, 
twisted by $h\in H^1(\Si(L)-\pi^{-1}(L);\Z/2\Z)$, where the value of $h$ on a curve $\gamma$ is the parity of half the sum of the linking numbers of 
$\pi\circ\gamma$ about the components of $L$ (following Turaev \cite[\S2.2]{Tu88}). Suppose that the crossing in Figure \ref{f:skein} is positive as, for example, illustrated in 
Figure \ref{f:orskein}, so that the orientation $o$ induces an orientation $o_1$ on $L_1$. 

\begin{figure}[ht]
\labellist\hair 2pt
\pinlabel $L_1$ at 40 -10
\pinlabel $L$ at 193 -10
\pinlabel $L_0$ at 354 -10
\endlabellist
\centering
\includegraphics[scale=0.6]{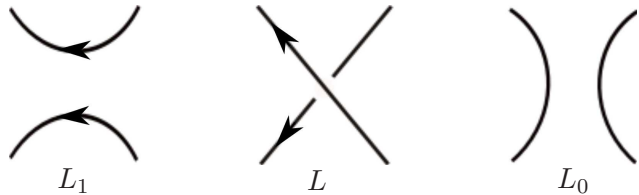}
\vspace{0.2cm}
\caption{The oriented link $(L,o)$ together with the oriented resolution $(L_1,o_1)$ and the unoriented resolution $L_0$.}
\label{f:orskein}
\end{figure}

Then, we can compute from Figure \ref{f:lambdas} that $h(\tla_1)=0$ and $h(\tla_0)=1$.  The Spin structure on $S^3$ restricts to the bounding structure on each of $\lambda_0$ and $\lambda_1$ using 
the $0$--framing. The map $\pi$ restricts to a diffeomorphism on neighbourhoods of $\tla_0$ and $\tla_1$. Therefore, the restriction of $\tilde\s$ to each of $\tla_0$ and $\tla_1$ using the 
pullback of the $0$--framing is also the bounding structure.  Also note that the preimage under $\pi$ of a disk bounded by $\lambda_i$ is an annulus with core $C$, so the framing of 
$\tla_i$ given by $C$ is the same as the pullback of the $0$--framing.

The spin structure $\t_{(L,o)}$ is equal to $\tilde\s$ twisted by $h$.  Since $\tilde\s$ restricts to the bounding spin structure on $\tla_1$, and $h(\tla_1)=0$, we see that $\t_{(L,o)}$ restricts to the bounding 
Spin structure on $\tla_1$ using the framing given by $C$.  On the other hand since $h(\tla_0)=1$, $\t_{(L,o)}$ restricts to the Lie Spin structure on $\tla_0$, 
again using the framing given by $C$.
It follows that $\t_{(L,o)}$ admits a unique extension $\s_o$ over the 2--handle giving the cobordism $V$, and does not extend over the cobordism $W$.

The restriction of $\s_o$ to $\Si(L_1)$ coincides with $\t_{(L_1,o_1)}$ outside of the solid torus $\tB$, and therefore also on the closed manifold $\Si(L_1)$.
\end{proof}

\section{Relations between correction terms}\label{s:correction}

By~\cite[Proposition~2.1]{OSz05} we have the following exact triangle:
\[
\begin{graph}(6,2)
\graphlinecolour{1}\grapharrowtype{2}
\textnode {A}(1,1.6){$\wHF(\Si(L_1))$}
\textnode {B}(5, 1.6){$\wHF(\Si(L))$}
\textnode {C}(3, 0.1){$\wHF(\Si(L_0))$}
\diredge {A}{B}[\graphlinecolour{0}]
\diredge {B}{C}[\graphlinecolour{0}]
\diredge {C}{A}[\graphlinecolour{0}]
\freetext (3,1.9){$F_V$}
\freetext (4.5,0.7){$F_W$}
\end{graph}
\]
where the maps $F_V$ and $F_W$ are induced by the surgery cobordisms of~\eqref{e:surgcobs}.
(All the Heegaard Floer groups are taken with $\Z/2\Z$ coefficients.) 

By~\cite[Proposition~3.3]{OSz05} (and notation as in that paper), if $L\subset S^3$ is a quasi--alternating link and $L_0$ and $L_1$ 
are resolutions of $L$ as in Definition~\ref{d:qa} then $\Si(L)$, $\Si(L_0)$ and $\Si(L_1)$ are $L$--spaces. Moreover, by assumption 
we have 
\begin{equation}\label{e:det}
|H^2(\Si(L);\Z)| = |H^2(\Si(L_0);\Z)| + |H^2(\Si(L_1);\Z)|.
\end{equation}
Since for every $L$--space $Y$ we have $|H^2(Y;\Z)| = \dim \wHF(Y)$, the Heegaard Floer surgery exact triangle reduces 
to a short exact sequence:
\begin{equation}\label{e:exaseq}
0 \to \wHF(\Si(L_1)) \xra{F_V} \wHF(\Si(L)) \xra{F_W} \wHF(\Si(L_0))\to 0.
\end{equation}
The type of argument employed in the proof of the following proposition goes back to~\cite{LSII07} 
and was also used in~\cite{St08}.

\begin{prop}\label{p:correction} 
Let $L$ be a quasi--alternating link and let $L_0$, $L_1$ be resolutions of $L$ as in Definition~\ref{d:qa}.  Let $o$ be an orientation on 
$L$ for which the crossing of Figure~\ref{f:skein} is positive, and let $o_1$ be the induced orientation on $L_1$. 
Then, the following holds: 
\begin{equation*}
-4d(\Si(L),\t_{(L,o)}) = -4d(\Si(L_1),\t_{(L_1,o_1)}) - 1. 
\end{equation*}
\end{prop}

\begin{proof} 
Since $\Si(L)$, $\Si(L_1)$ and $\Si(L_0)$ are $L$--spaces, we may think of the Spin$^c$ 
structures on these spaces as generators of their $\wHF$--groups, and we shall 
abuse our notation accordingly.
Let $V: \Si(L_1)\to\Si(L)$ be the surgery cobordism of~\eqref{e:surgcobs}, and let $\s_o$ be the unique Spin structure on $V$ which extends $\t_{(L,o)}$ as in 
Proposition~\ref{p:spincob}. Recall that, by definition, the map $F_U$ associated to a cobordism $U: Y_1\to Y_2$ is given by 
\[
F_U = \sum_{\s\in{\rm Spin}^c(U)} F_{U,\s}, 
\]
where $F_{U,\s}: \wHF(Y_1,\t_1)\to\wHF(Y_2,\t_2)$ and $\t_i = \s|_{Y_i}$ for $i=1,2$. We claim that 
\begin{equation}\label{e:claim}
F_{V,\s_o} (\t_{(L_1,o_1)}) = \t_{(L,o)}.
\end{equation}
The Heegaard Floer $\wHF$--groups admit a natural involution, usually denoted by ${\mathcal J}$.
The maps induced by cobordisms are equivariant with respect to the $\Z/2\Z$--actions associated to conjugation 
on Spin$^c$ structures and the ${\mathcal J}$--map on the  Heegaard Floer groups,  in the sense that,  
if $\bar{x}:= {\mathcal J}(x)$ for an element $x$, we have 
\begin{equation}\label{e:conjug}
F_{W,\bar{\s}} (\bar{x}) = \bar{F_{W,\s}(x)}
\end{equation}
for each $\s\in{\rm Spin}^c(W)$. Since by Proposition~\ref{p:spincob} there are no Spin structures on the surgery cobordism $W:\Si(L)\to\Si(L_0)$ of~\eqref{e:surgcobs}
which restrict to $\t_{(L,o)}$, the element $F_{W}(\t_{(L,o)})\in\wHF(\Si(L_0))$ has no Spin component. In fact, 
since $\t_{(L,o)}$ is fixed under conjugation and we are working over $\Z/2\Z$, \eqref{e:conjug} implies that the contribution of each  
non--Spin $\s\in{\rm Spin}^c(W)$ to a Spin component of $F_{W}(\t_{(L,o)})$ is cancelled by the contribution 
of $\bar{\s}$ to the same component. Therefore we may write 
\[
F_{W}(\t_{(L,o)}) = x + \bar{x}
\]
for some $x\in\wHF(\Si(L_0))$. By the surjectivity of $F_{W}$ there is some $y\in\wHF(\Si(L))$ with $F_{W}(y)=x$, therefore 
$F_{W}(\t_{(L,o)} + y + \bar{y}) = 0$, and by the exactness of~\eqref{e:exaseq} we have $\t_{(L,o)} + y + \bar{y} = F_{V}(z)$ 
for some $z\in\wHF(\Si(L_0))$. Since $F_{V}(\bar{z}) = \bar{F_{V}(z)} = F_{V}(z)$, the injectivity of $F_{V}$ implies  
$z=\bar{z}$. Moreover, $z$ must have some nonzero Spin component, otherwise we could 
write $z=u+\bar{u}$ and 
\[
F_{V}(u+\bar{u}) = \bar{F_{V}(u)} + \bar{F_{V}(\bar{u})} =  \bar{F_{V}(u)} + F_{V}(u)
\] 
could not have the Spin component $\t_{(L,o)}$.  This shows that there is a Spin structure $\t\in\wHF(\Si(L_1))$
such that $F_{V}(\t) = \t_{(L,o)}$. But, as we argued before for $F_{W}(\t_{(L,o)})$, in order for $F_{V}(\t)$ to have 
a Spin component it must be the case that there is some Spin structure $\s$ on $V$ such that $F_{V,\s}(\t)=\t_{(L,o)}$. 
Applying Proposition~\ref{p:spincob} we conclude $\s=\s_o$ and therefore $\t=\t_{(L_1,o_1)}$. 
This establishes Claim~\eqref{e:claim}. 

Using Equation~\eqref{e:det} and the fact that $\det(L_1)>0$ it is easy to check 
that $V$ is negative definite. The statement follows immediately from Equation~\eqref{e:claim} and the degree--shift 
formula in Heegaard Floer theory~\cite[Theorem~7.1]{OSz06} using the fact that $c_1(\s_o)=0$, $\si(V)=-1$ and $\chi(V)=1$. 
\end{proof} 


\section{The main result and a corollary}\label{s:final}

\main*

\begin{proof}
The statement trivially holds for the unknot, because the unknot has zero signature and 
the two--fold cover of $S^3$ branched along the unknot is $S^3$, whose only correction term vanishes. 
If $L$ is not the unknot and $L$ is quasi--alternating, there are quasi--alternating links $L_0$ and $L_1$ such that 
$\det(L) = \det(L_0) + \det(L_1)$ and $L$, $L_0$ and $L_1$ are related as in Figure~\ref{f:skein}. 
To prove the theorem it suffices to show that if the statement holds for $L_0$ and $L_1$ 
then it holds for $L$ as well. 

Denote by $L^m$ the mirror image of $L$, and by $o^m$ the orientation on $L^m$ naturally induced by 
an orientation $o$ on $L$.  The orientation--reversing diffeomorphism from $S^3$ to itself taking $L$ to $L^m$ 
lifts to one from $\Si(L)$ to $\Si(L^m)$  sending $\t_{(L,o)}$ to $\t_{(L^m,o^m)}$.  
Thus by~\cite[Theorem~8.10]{Li97} and~\cite[Proposition~4.2]{OSz03} we have  
\[
\si(L^m, o^m) = - \si(L,o)\quad\text{and}\quad 4d(\Si(L^m),\t_{(L^m,o^m)}) = 4d(-\Si(L),\t_{(L,o)}) = -4 d(\Si(L), \t_{(L,o)}), 
\]
therefore Equation~\eqref{e:main2} holds for $(L,o)$ if and only if it holds for $(L^m, o^m)$. Hence, without loss 
of generality we may now fix an orientation $o$ on $L$ so that the crossing appearing in 
Figure~\ref{f:skein} is positive. 

Denote by $o_1$ the orientation on $L_1$ naturally induced by $o$. By~\cite[Lemma~2.1]{MO07}
\begin{equation}\label{e:signature}
\si(L,o) = \si(L_1,o_1) - 1. 
\end{equation}
Since we are assuming that  the statement holds for $L_1$, we have 
\begin{equation}\label{e:correction}
\si(L_1,o_1) = -4 d(\Si(L_1), \t_{(L_1,o_1)}).
\end{equation}
Equations~\eqref{e:signature} and~\eqref{e:correction} together with  Proposition~\ref{p:correction} immediately 
imply Equation~\eqref{e:main2}.
\end{proof}

\begin{cor}\label{c:jones-turaev}
Let $(L,o)$ be an oriented, quasi--alternating link. Then, 
\begin{equation*}
 \tau(\Si(L),\t_{(L,o)}) = -\frac1{12} \frac{V'_{(L,o)}(-1)}{V_{(L,o)}(-1)},
\end{equation*}
where $\tau$ is Turaev's torsion function and $V_{(L,o)}(t)$ is the Jones polynomial of $(L,o)$.
\end{cor} 

\begin{proof} 
By~\cite[Theorem~3.4]{Ru04} we have
\begin{equation}\label{e:rusta}
d(\Si(L),\t_{(L,o)}) = 2\chi(HF^+_{\rm red} (\Si(L))) + 2 \tau(\Si(L),\t_{(L,o)}) - \lambda(\Si(L)), 
\end{equation}
where $\la$ denotes the Casson--Walker invariant, normalized so that it takes value $-2$ on the Poincar\'e sphere 
oriented as the boundary of the negative $E_8$ plumbing. Moreover, since $L$ is quasi--alternating 
$\Si(L)$ is an $L$--space, therefore the first summand on the right--hand side of~\eqref{e:rusta} vanishes. 
By~\cite[Theorem~5.1]{Mu93}, when $\det(L) > 0$ we have
\begin{equation}\label{e:mullins} 
\la(\Si(L)) = -\frac1{6}\frac{V'_{(L,o)}(-1)}{V_{(L,o)}(-1)} + \frac14\si(L,o)),
\end{equation}
Therefore, when $(L,o)$ is an oriented quasi--alternating link, Theorem~\ref{t:main2} together with 
Equations~\eqref{e:rusta} and~\eqref{e:mullins} yield the statement.
\end{proof} 

\bibliographystyle{amsplain}
\bibliography{biblio}
\end{document}